\documentclass[11 pt]{amsart}
\usepackage{graphicx}
\usepackage[all]{xy}
\usepackage[english]{babel}
\usepackage{amssymb,amsmath}
\usepackage{cite}
\usepackage{hyperref}

\usepackage{tikz}

\theoremstyle{cupthm}
\newtheorem{thm}{Theorem}[section]

\newtheorem{cor}[thm]{Corollary}
\newtheorem{lemma}[thm]{Lemma}
\theoremstyle{cupdefn}

\theoremstyle{cuprem}

\numberwithin{equation}{section}
\newtheorem{conj}[thm]{Conjecture}
\newtheorem{example}[thm]{Example}

\theoremstyle{definition}

\numberwithin{equation}{section}

\newcommand{\ff}{\mathbb F}

\newcommand{\pp}{\mathbb P}

\newcommand{\rf}[2]{{#1 \ref{#2}}}

\newcommand{\alpz}{\alpha_0}
\newcommand{\alpu}{\alpha_1}
\newcommand{\alpd}{\alpha_2}

\newcommand{\mor}{\operatorname{Mor}}

\newcommand{\rank}{\operatorname{rank}}

\newcommand{\Conceicao}{Concei{\c c}\~ao}

\begin{document}
\title[On integral points on isotrivial elliptic curves over function fields]{On integral points on isotrivial elliptic curves over function fields}
\author{Ricardo \Conceicao}
\address{Gettysburg College. 300 N Washington st. Gettysburg, PA 17325.}
\email{rconceic@gettysburg.edu}

\begin{abstract}
 Let $k$ be a finite field and $L$ be the  function field of a curve $C/k$ of genus $g\geq 1$.

In the first part of this note, we show that the number of separable $S$-integral points on a constant elliptic  curve $E/L$ is bounded  solely in terms of $g$, the size of $S$ and the rank of the Mordell-Weil group $E(L)$.

In the second part, we assume that $L$ is the function field of a hyperelliptic curve $C_A:s^2=A(t)$, where $A(t)$ is a square-free $k$-polynomial of odd degree. If $\infty$ is the place of $L$ associated to the point at infinity of $C_A$, then we prove that the set of separable $\{\infty\}$-points can be bounded solely in terms of $g$ and does not seem to depend on the  Mordell-Weil group $E(L)$. This is done by bounding  the number  of separable integral points over $k(t)$ on elliptic curves of the form $E_A:A(t)y^2=f(x)$, where $f(x)$ is a polynomial over $k$. Additionally, we show that, under an extra condition on $A(t)$, the existence of a separable integral point of ``small'' height on the elliptic curve $E_A/k(t)$ determines the isomorphism class of the elliptic curve $y^2=f(x)$.\end{abstract}


\maketitle
\section{Introduction}

Let $k$ be a finite field   and $L$ be the function field of a curve $C/k$. The goal of this note is to discuss arithmetical properties satisfied by integral points on isotrivial elliptic curves over $L$, i.e., when the $j$-invariant of the elliptic curve is an element of $k$. More specifically, we study integral points on constant elliptic curves and some of their quadratic twists. 

The first property is related to a longstanding conjecture of S. Lang that roughly says that the number of integral points is bounded independently of the model, for a certain class of models. To make this statement more precise, we let $L$ be a number field, $S$ be a finite set of places of $L$ containing the archimedian places, $R_S$ be the ring of $S$-integers of $L$, and let $E$ be an elliptic curve over $L$. 
\begin{conj}[Lang]
The number of $S$-integral points on a quasi-minimal model of an elliptic curve  $E/L$ is bounded solely in terms of the field $L$, the
set $S$, and the rank of the Mordell-Weil group $E(L)$. 
\end{conj}
For more information on this conjecture, including the definition of the quasi-minimal model of an elliptic curve, we refer the reader to the introduction of \cite{hindry_canonical_1988}.

Hindry-Silverman \cite{hindry_canonical_1988} show that Lang's conjecture is a consequence of the celebrated Szpiro's conjecture. Moreover, they prove that Lang's conjecture is true unconditionally if $L$ is the function field of a curve over a field of characteristic zero and $E/L$ is non-constant. In Section \ref{sec:langconj}, we prove the following theorem and we explain how it can be seen as a version of Lang's conjecture for constant elliptic curves over function fields.
\begin{thm}\label{thm:langconj}
Let $E/k$ be an elliptic curve, $C/k$ be a curve of genus $g\geq 1$,  and $S\subset C$ be a finite non-empty set of points. 
Then the number of non-constant separable $k$-morphisms $f :C\longrightarrow E$ satisfying $f^{-1}(O)\subset S$ is bounded by 
\[
\left(2\sqrt{|S|+4(g-1)}+1\right)^r
\]
where $r=\rank \mor_k(C,E)$ is the rank of the group of $k$-morphisms from $C$ to $E$.
\end{thm}

In the second part of this paper, we let $A(t)$ be a square-free polynomial of odd degree $d > 1$ over a finite field $k$ of odd characteristic. We  write $\infty$ for the point at infinity of the curve $C_A:y^2=A(t)$. Let $f(x)$ be a cubic polynomial over $k$ defining an elliptic curve $E:y^2=f(x)$ with point at infinity $O$. We prove in Corollary \ref{cor:boundquad} that the number of non-constant separable $k$-morphisms $f:C_A\longrightarrow E$  satisfying $f^{-1}(O)\subset\{\infty\}$ is bounded above by
$$
|k|^{2d-3}.
$$
Notice that, a priori, this bound does not depend directly on the rank of $\mor_k(C_A,E)$. 

To prove that the previous bound holds, in Section \ref{sec:quadtwi} we consider integral points over $k(t)$ on elliptic curves of the form $E_A:A(t)y^2=f(x)$. Using elementary methods,  we  prove that  if $P=(F,G)$ is a separable integral point on $E_A$ then
$$
\deg F<\deg A-1.
$$

Additionally, in Section \ref{sec:quadtwi} we show that $E_A$ can have a separable integral point of much lower degree only for certain curves E. Indeed, Theorem \ref{eqiv} shows that if $\deg A'(t)=0$ and $P=(F,G)$ is a separable integral point on $E_A$ satisfying 
$$
\deg F\leq \frac{\deg A-1}{2},
$$
then $j(E)=1728$.

\section{Lang's conjecture for constant elliptic curves}\label{sec:langconj}

We start this section by explaining why Theorem \ref{thm:langconj} is a version of Lang's conjecture for constant elliptic curves over finite fields. At the end of the section, we provide a proof of this theorem.

Let $E$ be an elliptic curve defined over a finite field $k$, let $L=k(C)$ be the function field of a curve $C$ of genus $g\geq 1$ and let $S\subset C$ be a finite non-empty set of points. Recall that our goal is to bound the number of $S$-integral points of $E$ in terms solely of $L$, $S$ and $\rank E(L)$.

The set $\mor_{k}(C,E)$  of $k$-morphisms from $C$ to $E$ is an abelian group canonically isomorphic to the Mordell-Weil group $E_0(L)$, where $E_0=E\times_k L$ (see \cite[Proposition 6.1]{ulmer_park_city_lect_2011}). Under this isomorphism, if $O\in E(k)$ is the point at infinity then the $k$-morphisms $f:C \longrightarrow E$ satisfying $f^{-1}(O) \subset S$ correspond to $S$-integral points on $E_0/L$. A $k$-morphism satisfying this condition is called  \emph{$S$-integral}.
 
In this setting, the set of  $S$-integral   morphisms is not finite.  Indeed, if $\phi:E\longrightarrow E$ is the Frobenius endomorphism on $E$ and $f$ is an $S$-integral morphism, then for every integer $n\geq 0$, the $k$-morphism $g_n=\phi^n\circ f$ is $S$-integral. To avoid such pathological examples, when discussing $S$-integral morphisms we disregard those that are inseparable.
  
Also, we assume that all of our $S$-integral morphisms are non-constant for the following reason. Notice that, with the exception of the constant morphism with value $\infty$,  all constant morphisms in $\mor_{k}(C,E)$ are $S$-integral.  Moreover, under the isomorphism $ E_0(L)\cong \mor_k(C,E)$, the set of constant morphisms $\mor^0_{k}(C,E)$ satisfies $E_0(k)\cong \mor^0_{k}(C,E)$. Therefore, by the Hasse-Weil theorem  the number of $S$-integral  morphisms that are constant is bounded by the size of $k$. Thus to prove Lang's conjecture for constant elliptic curves, we only need to bound the number of non-constant separable $S$-integral  morphisms in terms of $L$, $S$ and $\rank \mor_{k}(C,E)$.

Recall that the degree map, $\deg: \mor_k(C,E) \longrightarrow \mathbb{Z}$, defines a positive definite quadratic form on $\mor_k(C,E)/\mor^0_k(C,E)$,  which can be extended to a quadratic form on the real vector space $\mor_k(C,E)\otimes \mathbb{R}$. Additionally,  $\mor_k(C,E)/\mor^0_k(C,E)$ is a lattice in $\mor_k(C,E)\otimes \mathbb{R}$. 
This fact and the next result are the last ingredients needed  in our proof of Theorem \ref{thm:langconj}.

\begin{lemma}\label{latt}
Let $V$ be a $\mathbb{R}$-vector space of dimension $r$, $\Lambda\subset V$ be a lattice and $q:V\longrightarrow \mathbb{R}$ be a positive definite quadratic form on $V$. 

If $T$ is a positive real number   then
\[
\left|\{x \in \Lambda :q(x)\leq T\}\right| \leq \left(2\sqrt{\frac{T}{\lambda}}+1\right)^r
\]
for $\lambda=\min \{q(x):x \in \Lambda, x\neq 0\}$.
\end{lemma}
\begin{proof}
Let $\Lambda(T)=\{x \in \Lambda :q(x)\leq T\}$, for a fixed real number $T>0$. Suppose that $a$ and $b$ are distinct elements of $\Lambda(T)$ such that $\overline{a}=\overline{b}$ in $\Lambda/n\Lambda$, for some positive integer $n$. Therefore
there exists a non-zero $u\in \Lambda$ such that $a-b=nu$. As a consequence, if $\lambda=\min \{q(x):x \in \Lambda, x\neq 0\}$ then
\[
n^2\lambda \leq n^2q(u) =q(nu)=q(a-b)\leq 2q(a)+2q(b)\leq 4T,
\]
and
\[
n\leq \sqrt{\frac{4T}{\lambda}}.
\]

Hence, if we choose $n$ such that $\sqrt{{4T}/{\lambda}} +1\geq n>\sqrt{{4T}/{\lambda}}$, then the set $\Lambda(T)$ will inject into  $\Lambda/n\Lambda$. This implies
\[
\left|\{x \in \Lambda :q(x)\leq T\}\right| \leq |\Lambda/ n\Lambda| \leq n^r\leq \left(\sqrt{\frac{4T}{\lambda}}+1\right)^r
\qedhere \]
\end{proof}

\begin{proof}[Proof of Theorem \ref{thm:langconj}.]
Let  $f:C\longrightarrow E$ be a non-constant separable map satisfying $
f^{-1}(O)\subset S$. Let $e_{f}(P)$ denote the ramification index of $f$ at a point $P\in C
$ and denote by $R_{f}$ the support of the ramification divisor of $f$. Then
the Riemann-Hurwitz formula shows that 
\[
2g-2\geq \sum\limits_{P\in R_{f}}(e_{f}(P)-1)=\sum\limits_{P\in
R_{f}}e_{f}(P)-|R_{f}|\geq 2|R_{f}|-|R_{f}|=|R_{f}|,
\]
and
\[
\sum\limits_{P\in R_{f}}e_{f}(P)\leq 2g-2+|R_{f}|\leq 4(g-1).
\]
Thus
\begin{eqnarray*}
\deg f & =& \sum\limits_{P\in f^{-1}(O)}e_{f}(P)=\sum\limits_{P\in f^{-1}(O)\cap
R_{f}^{c}}1+\sum\limits_{P\in f^{-1}(O)\cap R_{f}}e_{f}(P)\\
& \leq & |S|+\sum\limits_{P\in R_{f}}e_{f}(P)\leq |S|+4(g-1).
\end{eqnarray*}

This shows that a non-constant separable morphisms $f:C\longrightarrow E$ satisfying $f^{-1}(O)\subset S$ is contained in the set
\[
\{ f\in \mor_k(C,E)/\mor^0_k(C,E): \deg f \leq |S|+4(g-1)\}.
\]

If we let $V=\mor_k(C,E)\otimes \mathbb{R}$, $\Lambda=\mor_k(C,E)/\mor^0_k(C,E)$ and $q=\deg$ then  \rf{Lemma}{latt} 
shows that the number of  non-constant separable $S$-integral morphisms is bounded by
\[
\left(2\sqrt{\frac{|S|+4(g-1)}{\lambda}}+1\right)^r
\]
where $\lambda=\min\{\deg f : f \in \mor_k(C,E)\backslash \mor^0_k(C,E)\}$. The result follows by noticing that $\lambda\geq 1$.
\end{proof}

We make the following remarks regarding Theorem \ref{thm:langconj}. 
\begin{itemize}
 \item For constant elliptic curves, $r\leq 4g$  (see \cite[10.1]{ulmer_elliptic_2002}). When  combined with the above result, we obtain a bound on the number of integral points on constant elliptic curves in terms solely of the genus of $C$ and $|S|$.
 \item Notice that one can improve the upper bound given in Theorem \ref{thm:langconj} if one can decrease the upper bound on the degree of non-constant separable $S$-integral morphisms or if one can find a non-trivial lower bound for $\min\{\deg f : f \in \mor_k(C,E)\backslash \mor^0_k(C,E)\}$.
\end{itemize}

\section{Integral points on Quadratic twists}\label{sec:quadtwi}

{\it Notation:} Let $k$ be a finite field of odd characteristic.
Let $A(t)$ be a square-free polynomial defined over $k$ of odd degree $d>1$ and let $C_A$ denote the curve defined by $s^2=A(t)$. We let $E/k$ be an elliptic curve defined by $y^2=f(x)$, for some cubic polynomial $f(x)$. Let $O$ and $\infty$ be the points at infinity of $E$ and $C_A$, respectively.

\subsection{Bounding separable integral points on constant elliptic curves over function fields of hyperelliptic curves.}

As discussed in Section \ref{sec:langconj}, the set of non-constant separable $k$-morphism $f:C_A\longrightarrow E$ satisfying $f^{-1}(O)\subset\{\infty\}$ can be thought of ``integral points'' on the elliptic curve $E$ over $L$, the function field of $C_A$. Theorem \ref{thm:langconj} shows that the number of such morphisms can be bounded in terms of $g=(d-1)/2$ and $\rank \mor_k(C_A,E)$. In this section, we give an upper bound (see Corollary \ref{cor:boundquad}) that depends only on $d$ and the size of $k$.

To obtain this new bound, we relate the set of $\infty$-integral $k$-morphisms to integral points on a quadratic twist of $E$. We let  $E_A$  be the elliptic curve defined over $k(t)$ by $A(t)y^2=f(x)$. An \emph{integral point} $(F,G)$ on $E_A$ is a point such that $F,G\in k[t]$.
\begin{lemma}\label{lem:intmap}
The set of  non-constant integral points on $E_A$ is in  bijection with the set of  non-constant $k$-morphism $\phi:C_A\longrightarrow E$  satisfying $\phi^{-1}(O) \subset \{\infty \}$. Moreover,  integral points $(F,G)$ with $F'\neq 0$ correspond to non-constant separable $k$-morphisms, and vice-versa.
\end{lemma}
\begin{proof}
Clearly, the map
$$
(F(t),G(t)) \longmapsto\phi(s,t)=(F(t),sG(t))
$$ 
defines a bijection between the set of integral points  on $E_A$ and the set of $k$-morphisms $\phi:C_A \longrightarrow E$ of the form 
\begin{equation}\label{eq:formphi}
\phi(s,t)=(F(t),sG(t)), 
\end{equation} 
for some polynomials $F(t)$ and $G(t)$. Also, a morphism of this form satisfies  $\phi^{-1}(O)\subset\{\infty\}$. Thus, we are left to show that  any $k$-morphism $\phi:C_A \longrightarrow E$ satisfying $\phi^{-1}(O)\subset \{\infty\}$  is given by \eqref{eq:formphi}.

Let $\sigma(t,s)=(t,-s)$ be the hyperelliptic involution of $C_A$. Using the group law on $E$, we  define the morphism $\phi\circ \sigma + \phi: C_A \longrightarrow E$ which is invariant under the action of the group generated by $\sigma$. Hence $\phi\circ \sigma + \phi$ factors through $\pp^1$, the quotient of $C_A$ by the group generated by $\sigma$. Since  a non-constant map from $\pp^1$ to $E$ does not exist,  $\phi^{-1}(O) \subset \{\infty\}$ implies that $\phi\circ \sigma + \phi= O$; that is, $\phi\circ \sigma =- \phi$.

Let us write $\phi(t,s)=(F_0(t,s),G_0(t,s))$, for some rational functions $F_0$ and $G_0$ of $k(C_A)$. The equation 
\[
(F_0(t,-s),G_0(t,-s))=\phi(t,-s) =\phi\circ\sigma =-\phi=(F_0(t,s),-G_0(t,s))
\] 
implies that $F_0(t,s)=F(t)$ is a rational function on $t$; and that $G_0(t,s)=sG(t)$, where $G(t)$ is  a rational function on $t$. Since $\phi^{-1}(O) \subset \{\infty\}$, we see that both $F(t)$ and $G(t)$ are polynomials, and $\phi$ has the desired form.

The ``moreover" part is proven by looking at the diagram of the function field extensions determined by \eqref{eq:formphi}
\[
 \xymatrix{   & &   k(t,s) \ar@{-}[dd]\\
 k(t) \ar@{-}[urr]|2  \ar@{-}[dd] &  &\\
  & &  k(x,y)  \\
 k(x) \ar@{-}[urr]|2 & &}
\]

Both degree 2 extensions are separable. Hence $ k(t,s)/ k(x,y)$ is separable if and only if $ k(t)/ k(x)$ is separable. The polynomial $F(T)-x \in  k(x)[T]$ is irreducible, so the extension $ k(t)/ k(x)$ is separable if and only if $(F(T)-x)'=F'(T)\neq 0$.
\end{proof}

In light of the previous result, we say that $(F,G)$ on $E_A$ is a \emph{separable integral point} if $F,G\in k[t]$ and $F'\neq 0$. In the next result we bound the ``height'' of such points.

\begin{lemma}\label{lem:gdf}
Let $(F,G)$ be a separable integral point on $E_A$. 
Then $G$ divides $F'$ and $d/3\leq \deg F < d-1$.
\end{lemma}
\begin{proof} An integral point $(F(t),G(t))$ on $E_A$ satisfies the identity
\begin{equation}\label{eq1}
A(t)G(t)^2 = f(F(t)).
\end{equation}
By equating degrees, we arrive at  $d \leq 3\deg F$.
 
 By differentiating  \eqref{eq1}, we are led to
\begin{equation}\label{eq2}
A'(t)G(t)^2 + 2A(t)G(t)G'(t) = F'(t)f'(F(t)),
\end{equation}

Let $\beta$ be a root of $G(t)$ of multiplicity $r$. By \eqref{eq1}, we have that $(t-\beta)^r$ divides $f(F(t))$ and, by (\ref{eq2}), we conclude that $(t-\beta)^r$ divides $F'(t)f'(F(t))$. Notice that $(t-\beta,f'(F(t)))=1$, since $f(x)$ has no repeated roots. Hence $(t-\beta)^r$  divides $F'(t)$ and, as a consequence, $G$ divides $F'$. Thus,  $\deg G \leq \deg F -1$ and, after comparing degrees in \eqref{eq1}, we arrive at $\deg F< d-1$.
\end{proof}

\begin{cor}\label{cor:boundquad}
 The number of  non-constant separable $k$-morphisms $\phi:C_A\longrightarrow E$  satisfying $\phi^{-1}(O) \subset \{\infty \}$ is bounded by
 $$
 |k|^{2d-3}.
 $$
\end{cor}
\begin{proof}
 By Lemma \ref{lem:intmap}, it is enough to count the number of integral points $(F,G)$ on $E_A$ with $F'\neq 0$. From Lemma \ref{lem:gdf}, we have $\deg G\leq \deg F-1$ and $\deg F<d-1$. Therefore, 
 $\deg G<d-2$ and the number  of integral points $(F,G)$ on $E_A$ with $F'\neq 0$ is at most
 $$
 |\{(F,G):F,G\in k[t],\deg F<d-1,\deg G<d-2\}|=|k|^{d-1}\cdot|k|^{d-2},
 $$
 as desired.
\end{proof}

\subsection{Integral points on quadratic twists and isomorphism classes.}

In Lemma \ref{lem:gdf}, we proved that for a separable integral point $(F,G)$ on $E_A:A(t)y^2=f(x)$ we have $d/3\leq \deg F<d-1$. In this section, we prove that if we assume the existence of a separable integral point $(F,G)$  with $d/3\leq \deg F\leq (d-1)/2$ then $j(E)=1728$, where $E$ is the elliptic curve defined by $y^2=f(x)$.

\begin{thm}\label{eqiv}
Suppose  $A'(t)\equiv \gamma \in \ff^*_q$. Let $E:y^2=f(x)$ be an elliptic curve defined over $ k$. Suppose $(F,G)$ is an integral point of $E_A/k(t)$  satisfying $F'\neq 0$. Then the following three conditions are equivalent:
\begin{enumerate}
\item[(A)] $2 \deg F \leq d-1$;
\item[(B)] $2 \deg G \leq \deg F-1$;
\item[(C)] $G^2=\beta F'$, for some $\beta\in  k^*$.
\end{enumerate}
 Furthermore, if one of the above  conditions is true  then $j(E)=1728$.
\end{thm}
\begin{proof}

From \eqref{eq1}, we know that $d+2\deg G=3\deg F$, and from this it easily follows that (A) is equivalent to (B). It is also clear that (C) implies (B), so all we need to show is that (B) implies (C).

Notice that, since both $F$ and $G$ are defined over $k$, a constant $\beta$ satisfying (C) is an element of $k$. Therefore, to prove that (B) implies (C) we may work over an extension of $k$ where $f(x)$ factors. 

We let $f(x)=(x-\alpz)(x-\alpu)(x-\alpd)$ and denote $F -\alpha_i$ by $F_i$, for $i\in\{0,1,2\}$. Then  
$$
f(F)=F_0F_1F_2,
$$ the  $F_i$'s are pairwise co-prime and
\begin{equation}\label{eq:fifj}
F_iF_j\equiv (\alpha_l -\alpha_i)(\alpha_l -\alpha_j)\mod F_l,
\end{equation}
for $\{i,j,l\}=\{0,1,2\}$.

By equating degrees in (\ref{eq1}), we obtain $\deg F_i\equiv d\equiv 1 \mod 2$. Consequently, by unique factorization and (\ref{eq1}), we find a non-constant polynomial $N_i$ satisfying  
\begin{equation}\label{eq:nidef}
\gcd(A,F_i)=N_i. 
\end{equation}

Since the  $F_i$'s are pairwise co-prime, we can find  a  polynomial $S_i$ such that
\begin{equation}\label{eq:sidef}
F_i=N_iS_i^2.
\end{equation}
We write  $s_i=\deg S_i$ and assume, without loss of generality, that 
\begin{equation}\label{ss}
s_0\geq s_1\geq s_2\geq 0.
\end{equation}
Also, observe that 
\begin{equation}\label{eq:gs}
 G=S_0S_1S_2,
\end{equation}
  and 
\begin{equation}\label{eq:ani}
 A=N_0N_1N_2.
\end{equation}

Given \eqref{eq:gs} and \eqref{eq:ani},  it follows from (\ref{eq2}) that
\begin{equation*}
 \gamma G^2 + 2N_0N_1N_2S_0S_1S_2(S'_0S_1S_2+S_0S'_1S_2+S_0S_1S'_2)=F'(F_0F_1+F_0F_2+F_1F_2).
\end{equation*}
Thus, from  \eqref{eq:sidef}, we get
\begin{equation*}\label{eq3}
 \gamma G^2 + 2N_0S_0S'_0F_1F_2+ 2N_1S_1S'_1F_0F_2+2N_2S_2S'_2F_0F_1=F'(F_0F_1+F_0F_2+F_1F_2).
\end{equation*}
For $l\in\{0,1,2\}$, this equality and \eqref{eq:fifj} imply
\begin{equation}\label{eq:gmod}
 G^2 + 2\beta_l N_lS_lS_l'\equiv \beta_l F'_l \mod F_l,
\end{equation}
where 
\begin{equation}\label{eq:betadef}
\beta_l= {(\alpha_l -\alpha_i)(\alpha_l -\alpha_j)}/{\gamma}. 
\end{equation}

Since 
\begin{equation}\label{fp}
F'=F'_i=N'_iS_i^2+2N_iS_iS'_i,
\end{equation}
we arrive from \eqref{eq:gmod} at
\begin{equation*}
 G^2 \equiv \beta_l N'_lS_l^2 \mod F_l.
\end{equation*}

Clearly,  $\deg(N'_lS_l^2)<\deg F_l=\deg F$. Therefore if (B) is true, we get  $\deg G^2 < \deg F$; and ultimately, 
\begin{equation}\label{eq4}
 G^2=\beta_l N'_lS_l^2
\end{equation}
for $l\in\{0,1,2\}$. 

Now consider $\{i,l\}=\{1,2\}$. Multiplying (\ref{fp}) by $\beta_i$ and using (\ref{eq4}), we obtain
\[
\beta_i F'= G^2 +2\beta_iN_iS_iS'_i.
\]

\rf{Lemma}{lem:gdf} implies that $G$ divides $2\beta_iN_iS_iS'_i$. Thus, from \eqref{eq:gs} we get
\[
S_0S_l \mid  2\beta_i N_iS'_i,
\]
since $(S_0S_l,N_i)=1$. This in turn implies 
\[
S_0S_l\mid 2\beta_iS'_i.
\]
Notice that $S'_i= 0$, for $i =1,2$, otherwise   (\ref{ss}) would imply
\[
s_i \leq s_0+s_l \leq s_i -1.
\]
Thus, \eqref{fp} reads as $$F'_i=N'_iS_i^2,$$ and \eqref{eq4} becomes
\[
 G^2=\beta_i N'_iS_i^2=\beta_i F'_i=\beta_iF'.
\]
This finishes the proof that (A), (B) and (C) are all equivalent.

To show the second part, let us assume that either one of the equivalent statements (A), (B) or (C) is true. Then the last equality shows that necessarily $\beta=\beta_1=\beta_2$, since $F'\neq 0$.

By performing a change of variable $x\longmapsto x+\alpha_0$, we obtain an elliptic curve isomorphic to $E$, and we may assume that $\alpz=0$. Therefore, from \eqref{eq:betadef} we arrive at
\[
 \frac{\alpha_1 (\alpha_1 -\alpha_2)}{\gamma} = \beta_1 = \beta_2 = \frac{\alpha_2 (\alpha_2 -\alpha_1)}{\gamma}.
\]
Thus,  $\alpu^2 =\alpd^2$.  Since the $\alpha_i$'s are all distinct, we have $\alpu =-\alpd\neq 0$. This shows that $E$ is isomorphic over $ k$ (or an extension of $k$) to $y^2=x^3-a^2x$, for $a=\alpd$. Since this last elliptic curve has $j$-invariant  1728, the result follows. 
\end{proof}

We give an example  to show that the hypothesis on Theorem \ref{eqiv} do  not give  vacuous conditions.
\begin{example}
 Let $k$ be a finite field of size $q\equiv 3 \mod 4$. Then
 $$
 (t^{(q-1)/2},t^{(q-3)/4})
 $$
 is a separable integral point on $(t^q-t)y^2=x^3-x$.
\end{example}

\bibliographystyle{amsalpha}

\end{document}